\documentclass[reqno,12pt]{amsart}
\usepackage[colorlinks=true,citecolor=blue]{hyperref}
\usepackage{mathptmx, amsmath, amssymb, amsfonts, amsthm, mathptmx, enumerate, color}
\usepackage{graphicx}
\usepackage{epstopdf}

\newcommand{\NN}{{\mathbb N}}

\def\bege{\begin{equation}} \def\ende{\end{equation}}
\def\begr{\begin{eqnarray}} \def\endr{\end{eqnarray}}
\usepackage{amsmath}
\usepackage{amssymb}
\usepackage{cite}

\usepackage{bbm}
\usepackage{amsmath}
\usepackage{txfonts}
\usepackage{stmaryrd}
\usepackage{amssymb}
\usepackage{amsfonts}
\usepackage{times}
\usepackage{mathrsfs}

\usepackage{amsthm}
\usepackage{enumerate}

\usepackage{framed}
\usepackage{lipsum}
\usepackage{color}

\newcommand{\TT}{{\mathbb T}}

\newcommand{\DD}{{\mathbb D}}
\def\B{\mathcal{B}}

\def\D{\mathbb{D}}

\def\a{\alpha}
\def\b{\beta}

\def\t{\theta}

\def\begr{\begin{eqnarray}} \def\endr{\end{eqnarray}}

\def\msk{\medskip}
\def\ol{\overline}
\newtheorem{Lemma}{Lemma}[section]
\newtheorem{Theorem}[Lemma]{Theorem}

\newtheorem{Proposition}[Lemma]{Proposition}
\newtheorem{Definition}[Lemma]{Definition}
\newtheorem{Remark}[Lemma]{Remark}

\newcounter{other}            

\begin{document}
	\title[Generalized integration operators]{Generalized integration operators on  analytic tent spaces}
	
	\author{Rong Yang, Lian Hu and Songxiao Li$^\dagger$  }
	
	\address{Rong Yang
	\\Institute of Fundamental and Frontier Sciences, University of Electronic Science and Technology of China, 610054, Chengdu, Sichuan, P.R. China.}
	\email{yangrong071428@163.com  }

\address{Lian Hu
	\\Institute of Fundamental and Frontier Sciences, University of Electronic Science and Technology of China, 610054, Chengdu, Sichuan, P.R. China.}
	\email{hl152808@163.com }

	\address{Songxiao Li\\ Department of Mathematics, Shantou University, 515063, Shantou, Guangdong, P.R. China. } \email{jyulsx@163.com}

	
	\begin{abstract} In this paper, the boundedness and compactness of generalized integration operators $T_g^{n,k}$ between different analytic tent spaces in the unit disc are completely characterized. 
		
		\thanks{$\dagger$ Corresponding author.}
	\end{abstract}
	
	\maketitle
	
\section{Introduction}
Let $\mathbb{D}$ denote the unit disc on the complex plane $\mathbb{C}$, $\mathbb{T}$ its boundary, and $H(\mathbb{D})$ be the set of
all analytic functions in $\mathbb{D}$. 
We denote by $\NN$ the set of positive integers.
The Hardy space $H^p(0<p<\infty)$ consists of all $f \in H(\mathbb{D})$ for which
	$$
	\|f\|_{H^p}^p=\sup _{0< r<1} \int_0^{2 \pi}|f(r e^{i \theta})|^p \frac{d\t}{2 \pi}<\infty.
	$$
 Let $0<p<\infty$ and $\alpha>-1$, the weighted Bergman space $A_\alpha^p$ consists of $f\in H(\D)$ such that
	$$\|f\|_{A_\alpha^p}^p=(\alpha+1) \int_{\mathbb{D}}|f(z)|^p(1-|z|^2)^\alpha d A(z)<\infty,$$ where $dA(z)=\frac{1}{\pi}dxdy$ is the normalized Lebesgue area measure on $\D$.

Let $\eta \in \mathbb{T}$ and $\zeta>\frac{1}{2}$. The non-tangential region $\Gamma_\zeta(\eta)$ is defined by
$\Gamma(\eta)=\Gamma_\zeta(\eta) =\{z \in \mathbb{D}:|z-\eta|<\zeta(1-|z|^2) \}.$	Let $0<p, q<\infty$ and $\a>-2$. The tent space $T_p^q(\alpha)$ consists of all measurable functions $f$ on $\mathbb{D}$ with
\begin{align*}
\|f\|_{T_p^q(\alpha)}=\left(\int_\mathbb{T}\left(\int_{\Gamma(\eta)}|f(z)|^p(1-|z|^2)^{\a} d A(z)\right)^{\frac{q}{p}}|d \eta|\right)^{\frac{1}{q}}<\infty .
\end{align*}
In particular, for $\a=0$, we write $T_p^q$ instead of $T_p^q(\a)$. 	
For $p=\infty$ and $0<q<\infty$, the tent space $T_{\infty}^q(\alpha)$ consists of all measurable functions $f$ on $\mathbb{D}$ such that
\begin{align*}
\|f\|_{T_\infty^q(\alpha)}=\left(\int_\mathbb{T} \left(\operatorname{esssup}_{z \in \Gamma(\eta)}|f(z)|\right)^q|d \eta|\right)^{\frac{1}{q}}<\infty. 
\end{align*}
Notice that the definition of $T^q_\infty(\alpha)=T^q_\infty$ is independent of $\a$. 
For $q=\infty$ and $0<p<\infty$, the tent space $T_p^{\infty}(\alpha)$ consists of all measurable functions $f$ on $\mathbb{D}$ for which
$$
\|f\|_{T_p^{\infty}(\alpha)}=\operatorname{esssup} _{\eta \in \mathbb{T}}\left(\sup _{u \in \Gamma(\eta)} \frac{1}{1-|u|^2} \int_{S(u)}|f(z)|^p(1-|z|^2)^{\a+1} d A(z)\right)^{\frac{1}{p}}<\infty,
$$
where
$S(r e^{i \theta})=\left\{\lambda e^{i t}: |t-\theta| \leq \frac{1-r}{2},1-\lambda \leq 1-r\right\}$ for $r e^{i \theta} \in \mathbb{D} \backslash\{0\}$ and $S(0)=\mathbb{D}$.
Denote $T_p^q(\alpha) \cap H(\mathbb{D})$ by $A T_p^q(\alpha)$.
	
Tent spaces were initially introduced by Coifman, Meyer and Stein  in \cite{cms} to study problems in harmonic analysis. They provided a  general framework for questions regarding some important spaces such as Hardy spaces and Bergman spaces. The aperture $\zeta  $ of the non-tangential region $\Gamma_\zeta(\eta)$   is suppressed in the above definition since   any two apertures generate the same function space with equivalent quasinorms.
By the non-tangential maximal function characterization of the Hardy space, $AT^q_\infty=H^q$, which can be see as the limit of $AT_p^q(\a)$ when $p\to\infty$ (see
\cite{z2}). In particular, $AT_p^p(\a-1)=A^p_\a$, the weighted Bergman space.

In \cite{ca1,ca2}, Carleson  characterized the positive Borel measures $\mu$ on the unit disc $\DD$ such that the identity  operator  $I_d:H^{p}  \to   L^{p}({d\mu})$ is bounded 
if and only if $\mu$ is a Carleson measure, 
that is,  
$\sup_{I\subset \mathbb{T}}\frac{\mu(S(I))}{|I|}<\infty$.
Here, $S(I)=\{ z\in\D:1-|I|\le|z|<1,\frac{z}{|z|}\in I  \}$.
In \cite{Dur}, Duren generalized Carleson's result and   descibed the boundedness of the identity  operator  $I_d:H^{p}  \to   L^{q}({d\mu})$ when $0<p\leq q<\infty$.
In \cite{l2}, Luecking investigated the positive Borel measure $\mu$ such that differential operator of order $m$ mapping $H^p$ to $L^q(d\mu)$ boundedly.
In \cite{l3}, Luecking characterized the positive Borel measure $\mu$ such that differential operator of order $m$ mapping the Bergman space $A^p$ and the tent space $AT_p^q(\a)$ to $L^q(d\mu)$ boundedly.
Subsequently, the Carleson measure was extended to higher dimensional cases and more general cases, and became a very important tool for the study of function space and operator theory.
See \cite{ag1,wz2} for more results of the differential operator of order $m$ mapping  the tent space $AT_p^q(\a)$ into $L^q(d\mu)$.

Let $g \in H(\mathbb{D})$, $k\in\NN\cup\{0\}$  and $n\in\NN$ such that $0 \leq k<n$. The generalized integration operator $T_g^{n, k}$ is defined by
$$
T_g^{n, k} f(z)=I^n\left(f^{(k)}(z) g^{(n-k)}(z)\right), \quad f \in H(\mathbb{D}).
$$
Here,  $I^n$ is the $n$-th iteration of the integration operator $I f(z)=\int_0^z f(t)dt$.
The operator $T_g^{n, k}$ was first introduced by Chalmoukis \cite{ch}. In particular, when $n=1$ and $k=0$, we have
$$
T_g^{1,0} f(z)=\int_0^z f(\eta) g^{\prime}(\eta) d \eta=T_g f(z) .
$$

In 1977, the Volterra integration operator $T_g$ was first introduced and stuided by Pommerenke in \cite{p}. It was proved that $T_g$ is bounded on $H^2$ if and only if $g\in BMOA$. 
In 1995, Aleman and Siskakis proved that $T_g$ is bounded  on $H^p$($p\ge1$) if and only if $g\in BMOA$ in \cite{as1}.
In 1997, they proved that $T_g$ is bounded  on $A^p$ if and only if $g\in\B$ in \cite{as2}.
Recall that the Bloch space $\mathcal{B}$ is the class of all $f \in H(\mathbb{D})$ with
$$
\|f\|_{\mathcal{B}}=|f(0)|+\sup _{z \in \mathbb{D}}(1-|z|^2)\left|f^{\prime}(z)\right|<\infty.
$$
See \cite{mppw,hyl,lly,llw,zz,w2} and the
references therein for more study of the operator $T_g$.
In \cite{ch},  Chalmoukis studied the boundedness of  the  operator $T_{g}^{n,k}$ between different Hardy spaces. For example, he showed that $T_g^{n,0}: H^p\to H^p$ is bounded if and only if $g\in BMOA$, $T_g^{n,k}: H^p\to H^p$ is bounded  if and only if $g\in\B$ when $k\geq 1$, while $T_g^{n,k}:H^p\to H^q$ is bounded  if and only if
$\sup\limits_{z\in\D}(1-|z|^2)^{\frac{1}{q}-\frac{1}{p}+n-k}|g^{(n-k)}(z)|<\infty$ when $0<p<q<\infty$. 
In \cite{dlq}, the authors characterized the boundedness, compactness and Schatten class  membership of  the operator $T_g^{n,k}$ on Bergman spaces induced by doubling weights.
From the result just mentioned, we see that $T_g^{n,k}$ has a very different behavior from that of $T_g$.

	Motivated by  \cite{ch,dlq,qz}, in this paper, we study the boundedness and compactness of the operator $T_g^{n,k}:AT_p^q(\a)\to AT_t^s(\b)$. The necessary and sufficient conditions of the operator $T_g^{n,k}$ to be boundedly and compactly are given.

	This paper is organzied as follows. In Section 2, we mainly provide some related properties of tent spaces. In Section 3 and Section 4, we describe the boundedness and compactness of $T_g^{n,k}:AT_p^q(\a)\to AT_t^s(\b)$, respectively.
	
	In this paper, the symbol $E\lesssim F$ means that there exists a positive finite constant $C$ such that $E\le CF$. The value of $C$ may change from one occurrence to another. The expression $E\asymp F$ means that both $E\lesssim F$ and $F\lesssim E$ hold.

	\section{Preliminaries}
	In this section, we state some notations and lemmas, which will be used in the proof of main results in this paper.

	\begin{Lemma}\cite[Lemma 4]{ar}\label{2.3}
		Let $0<p, q<\infty$ and $\lambda>\max \{1, \frac{p}{q}\}$. Then there are constants $C_1=C_1(p, q, \lambda, C)$ and $C_2=C_2(p, q, \lambda, C)$ such that
		$$
		C_1 \int_{\mathbb{T}} \mu(\Gamma(\eta))^{\frac{q}{p}}|d \eta| \leq \int_{\mathbb{T}}\left(\int_{\mathbb{D}}\left(\frac{1-|z|^2}{|1-z \bar{\eta}|}\right)^\lambda d\mu\right)^{\frac{q}{p}}|d \eta| \leq C_2 \int_{\mathbb{T}} \mu(\Gamma(\eta))^{\frac{q}{p}}|d \eta|
		$$
		for every positive measure $\mu$ on $\mathbb{D}$.
	\end{Lemma}


	\begin{Lemma}\label{Tfn}
	Let $0<p,q<\infty$, $\a>-2$ and $n\in\NN\cup\{0\}$. Then $f\in AT_p^q(\a)$ if and only if
	$$
	\int_{\mathbb{T}}\left(\int_{\Gamma{(\eta)}}|f^{(n)}(z)|^{p}(1-|z|^2)^{np+\a} dA(z)\right)^{\frac{q}{p}}|d\eta|<\infty.
	$$
	\end{Lemma}
	
	\begin{proof}
	Form \cite[Theorem 2]{pa}, we have
	$$
	\left\|f^{\prime}(\cdot)(1-|\cdot|)\right\|_{T_p^q(\alpha)} \asymp\|f\|_{T_p^q(\alpha)},\quad f \in H(\D),
	$$
	which get the desired result.
	\end{proof}

	\begin{Lemma}\cite[Lemma 2.3]{wz2}\label{le1}
	Let $0<p\le t<\infty$, $0<s\le q<\infty$, $\a>-2$, and $\b=\frac{t(\a+2)}{p}-2$. Then $AT_p^q(\a)\subset AT_t^s(\b)$.
	\end{Lemma}

	\begin{Lemma}\cite[Lemma 2.4]{wz2}\label{le2}
	If $0<q<s<\infty$, $0<p<\infty$, $\a>-2$, then $AT_p^q(\a)\subset A_\delta^s$ with bounded inclusion, where $\delta=\frac{s(\a+2)}{p}+\frac{s}{q}-2$.
	\end{Lemma}

	Let $\rho(z,w)=\left| \frac{z-w}{1-\ol{w}z} \right|$ and $\beta(z, w)$ be the hyperbolic metric on $\mathbb{D}$, that is,  
	$	\b(z,w)=\frac{1}{2}\log\frac{1+\rho(z,w)}{1-\rho(z,w)}$, $z,w\in\D.$
If $\beta(z, w)<r$ for any $z, w \in \mathbb{D}$ , then $ 1-|w| \asymp 1-|z|$.
Let $D(z, r)=\{w \in \mathbb{D}: \beta(z, w)<r\}$ be the hyperbolic disc of radius $r>0$ centered at $z \in \mathbb{D}$. 
The following lemma is very important in this paper and will be used frequently.
	
\begin{Lemma}\cite[Lemma 2.3]{w1}\label{2.9}
Let $\zeta>1, r \geq 0$ and $\eta \in \mathbb{T}$. If $\zeta_+=(\zeta+1) e^{2 r}-1$, then
$$D(z, r) \subset \Gamma_{\zeta_+}(\eta)$$ for all $z \in \Gamma_{\zeta}(\eta)$.
\end{Lemma}

\begin{Lemma}\label{fn}
Let $0<p,q<\infty$, $\a>-2$, $f\in AT_p^q(\a)$ and $n\in\NN\cup\{0\}$. Then 
$$
|f^{(n)}(z)|\lesssim\frac{\|f\|_{AT_p^q(\a)}}{(1-|z|^2)^{\frac{\a+2}{p}+\frac{1}{q}+n}},\quad z\in\D.
$$
\end{Lemma}
\begin{proof}
According to \cite[Proposition 4.13]{z1}, we have
$$
|f^{(n)}(z)|^p\lesssim\frac{1}{(1-|z|^2)^{\a+2}}\int_{D(z,r)}|f^{(n)}(w)|^p(1-|w|^2)^\a dA(w).
$$
For $f\in H(\D)$, using Lemma $\ref{2.9}$, we obtain
\begin{align*}
|f^{(n)}(z)|^q&\lesssim \left(\frac{1}{(1-|z|^2)^{\a+2+np}}\int_{D(z,r)}|f^{(n)}(w)|^p(1-|w|^2)^{np+\a} dA(w)\right)^\frac{q}{p}\\
&\lesssim\frac{1}{(1-|z|^2)^{\frac{q(\a+2)}{p}+nq+1}}\int_{\mathbb{T}}\left(\int_{\Gamma_{\zeta_+}(\eta)}|f^{(n)}(w)|^p(1-|w|^2)^{np+\a} dA(w) \right)^{\frac{q}{p}}|d\eta|\\
&\lesssim\frac{\|f\|_{T_p^q(\a)}}{(1-|z|^2)^{\frac{q(\a+2)}{p}+nq+1}},
\end{align*}
which implies the desired result.
\end{proof}

The sequence $Z=\{a_j\}$ is a separated sequence if there is a constant $\tau>0$ with $\beta(a_j, a_j) \geq \tau$ for $k \neq j$. 
For $r>\kappa>0$, the sequence $Z=\{a_j\}$ is said to be an $(r, \kappa)$-lattice if 
$\mathbb{D}=\bigcup_k D(a_j, r)$ and
the sets $D(a_j, \kappa)$ are pairwise disjoint.
It is obvious that any $(r, \kappa)$-lattice is a separated sequence.

\begin{Definition}
Let $0<p, q<\infty$ and $Z=\{a_j\}$ be an $(r, \kappa)$-lattice. The tent space $T_p^q(Z)$ consists of all $\{x_j\}$ such that
\begin{align*}
\|\{x_j\}\|_{T_p^q(Z)}=\left(\int_\mathbb{T}\left(\sum_{a_j \in \Gamma(\eta)}|x_j|^p\right)^{\frac{q}{p}}|d \eta|\right)^{\frac{1}{q}}<\infty.
\end{align*}
For $p=\infty$ and $0<q<\infty$, the tent space $T_\infty^q(Z)$ consists of all $\{x_j\}$ with
$$
\|\{x_j\}\|_{T_\infty^q(Z)}=\left(\int_\mathbb{T}\sup _{a_j \in \Gamma(\eta)}|x_j|^q|d \eta|\right)^{\frac{1}{q}}<\infty.
$$
For $q=\infty$ and $0<p<\infty$, the tent space $T^\infty_p(Z)$ consists of all $\{x_j\}$ for which
$$
\|\{x_j\}\|_{T_p^{\infty}(Z)}=\operatorname{esssup} _{\eta \in \mathbb{T}}\left(\sup _{u \in \Gamma(\eta)} \frac{1}{1-|u|^2} \sum_{a_j \in S(u)}|x_j|^p(1-|a_j|^2)\right)^{\frac{1}{p}}<\infty.
$$
\end{Definition}
	
Analogously, the aperture $\zeta  $ of the non-tangential region $\Gamma_\zeta(\eta)$   is suppressed in the above definition since any two apertures generate the tent spaces of sequences  with equivalent quasinorms. 

\begin{Remark}\label{remark}
Let	$Z=\{a_j\}$ be an $(r, \kappa)$-lattice.
By Lemma \ref{2.9}, there exists $\zeta_{+}>\zeta>\frac{1}{2}$ such that
\begin{align}\label{*}
\bigcup_{D(a_j, r) \bigcap \Gamma_\zeta(\eta) \neq \emptyset} D(a_j, r) \subset \Gamma_{\zeta_{+}}(\eta).
\end{align}
In the rest of this paper, we will denote  $\Gamma_{\zeta_{+}}(\eta)$ by $\widetilde{\Gamma}(\eta)$ in (\ref{*}) when we need not mention the aperture $\zeta_+$.
\end{Remark}

The following result is well known (see for example \cite{ag1}). 

\begin{Lemma}\label{2.13}
Let $0<p<\infty$. There is a constant $C>0$ such that
$$
\int_{\mathbb{T}}\left(\sup _{a_j \in \widetilde{\Gamma}(\eta)}|x_j|^p\right)|d \eta| \leq C \int_{\mathbb{T}}\left(\sup _{a_j \in \Gamma(\eta)}|x_j|^p\right)|d \eta|
$$
for all sequences $\{a_j\} \subset \mathbb{D}$ and $\{x_j\} \subset \mathbb{C}$.
\end{Lemma}

\begin{Proposition}\cite[Proposition 2.7]{ag1}\label{2.7}
Let $Z=\{a_j\}$ be an $(r, \kappa)$-lattice and $L \geq 1, R>0$. There is a positive integer $N=$ $N(L, R, Z)$ such that for each point $z \in \mathbb{D}$ there are at most $N$ hyperbolic discs $D(a_j, L r)$ satisfying $D(z, R) \cap D(a_j, L r) \neq \emptyset$.
\end{Proposition}

The following lemma is very important, and we can find it in \cite{pa}.
	
\begin{Proposition}\cite [Lemma 14]{pa}\label{2.23}
Let $0<p, q<\infty$, $\a>-2$, $L>\max \{1, \frac{p}{q}, \frac{1}{q}, \frac{1}{p}\}$, and
$Z=\{a_j\}$ be an $(r, \kappa)$-lattice. Then the operator $S: T_p^q(Z) \rightarrow A T_p^q(\alpha)$ is bounded. Here,
$$
S(\{x_j\})(z)=\sum_{j=1}^{\infty} x_j \frac{(1-|a_j|)^{L}}{(1-\overline{a_j} z)^{L+\frac{\a+2}{p}}}.
$$
\end{Proposition}

We will use the following proposition, which is a factorization  of tent spaces  of sequences.

\begin{Proposition}\cite[Proposition 2.16]{ag1}\label{2.16}
Let $0<p, q<\infty$ and $Z=\{a_j\}$ be an $(r, \kappa)$-lattice. If $p \leq p_1, p_2 \leq\infty, q \leq q_1$, $q_2 \leq\infty$ and satisfy $\frac{1}{p}=\frac{1}{p_1}+\frac{1}{p_2}$ and $\frac{1}{q}=\frac{1}{q_1}+\frac{1}{q_1}$. Then
$$
T_p^q(Z)=T_{p_1}^{q_1}(Z) \cdot T_{p_2}^{q_2}(Z).
$$
\end{Proposition}

The following propositions 
play a key role in the proof of our main theorem. 
From now on, $p'$ is the conjugate exponent of $p$ with $\frac{1}{p}+\frac{1}{p'}=1$.

\begin{Proposition}\cite[Lemma 6]{ar}\label{2.18}
Let $1 \leq p<\infty, 1<q<\infty$ and $Z=\{a_j\}$ be an $(r, \kappa)$-lattice . Then $\left(T_p^q(Z)\right)^* \cong T_{p^{\prime}}^{q^{\prime}}(Z)$. The isomorphism between $\left(T_p^q(Z)\right)^*$ and $T_{p^{\prime}}^{q^{\prime}}(Z)$ is given by the operator
$
\{d_j\} \mapsto\langle\cdot,\{d_j\}\rangle,
$
where $\langle\cdot,\{d_j\}\rangle$ is defined by
$$
\langle\{c_j\},\{d_j\}\rangle=\sum_j c_j d_j(1-|a_j|^2), \quad\{c_j\} \in T_p^q(Z) .
$$
In fact, $\|\{d_j\}\|_{T_{p^{\prime}}^{q^{\prime}}(Z)} \asymp \sup \left\{\left|\sum_j c_j d_j(1-|a_j|^2)\right|:\|\{c_j\}\|_{T_p^q(Z)}=1\right\}$.
\end{Proposition}

\begin{Proposition}\cite[Proposition 2]{ar}\label{2.19}
Let $0<p<1<q<\infty$ and $Z=\{a_j\}$ be an $(r, \kappa)$-lattice. Then $\left(T_p^q(Z)\right)^* \cong T_{\infty}^{q^{\prime}}(Z)$. The isomorphism between $\left(T_p^q(Z)\right)^*$ and $T_{\infty}^{q^{\prime}}(Z)$ is given by the operator as in Proposition \ref{2.18}. In fact, $$\|\{d_j\}\|_{T_{\infty}^{q^{\prime}}(Z)} \asymp \sup \left\{\left|\sum_j c_j d_j(1-|a_j|^2)\right|:\|\{c_j\}\|_{T_p^q(Z)}=1\right\}.$$
\end{Proposition}

\begin{Proposition}\cite[Lemma 3.4]{je}\label{2.20}
Let $1<p<\infty$ and $Z=\{a_j\}$ be an $(r, \kappa)$-lattice. Then $\left(T_p^1(Z)\right)^* \cong T_{p^{\prime}}^{\infty}(Z)$. The isomorphism between $\left(T_p^1(Z)\right)^*$ and $T_{p^{\prime}}^{\infty}(Z)$ is given by the operator as in Proposition \ref{2.18}.
In fact, $$\|\{d_j\}\|_{T^\infty_{p^{\prime}}(Z)} \asymp \sup \left\{\left|\sum_j c_j d_j(1-|a_j|^2)\right|:\|\{c_j\}\|_{T_p^1(Z)}=1\right\}.$$
\end{Proposition}
	
	
\begin{Proposition}\cite[Proposition 2.21]{ag1}\label{2.21}
If either $0<p<\infty, 0<q<1$ or $0<p \leq 1, q=1$, let $Z=\{a_j\}$ be an $(r, \kappa)$-lattice, then
$$
\sup \left\{\left|\sum_j c_j d_j(1-|a_j|^2)\right|:\|\{d_j\}\|_{T_p^q(Z)}=1\right\} \asymp \sup _k|c_j|(1-|a_j|^2)^{1-\frac{1}{q}}
$$
for any sequence $\{c_j\}$.
\end{Proposition}

We also need the following two lemmas.

\begin{Lemma}\cite[Theorem 1.1]{wz2}\label{th1}
Let $0<p,q,s<\infty$, $\a>-2$, $m\in\NN\cup\{0\}$ and let $\mu$ be a positive Borel measure on $\D$. Then differential type operator $D^{(m)}:AT_p^q(\a)\to L^s(\mu)$ is bounded if and only if any $r\in(0,1)$, one of the following conditions holds:
	
(a) If $q>s$ and $p>s$, then $\Psi_\mu\in T^{\frac{q}{q-s}}_{\frac{p}{p-s}}$;
	
(b) If $p\le s<q$, then $\Psi_\mu(\cdot)(1-|\cdot|^2)^{2(1-\frac{s}{p})}\in T^{\frac{q}{q-s}}_\infty$;
	
(c) If $q=s<p$, then $\Psi_\mu\in T^\infty_{\frac{p}{p-s}}$;
	
(d) If $p\le q=s$ or $q<s$, then $\Psi_\mu(\cdot)(1-|\cdot|)^{3-\frac{2s}{p}-\frac{s}{q}}\in L^\infty$.\\
Here, $\Psi_\mu(z)=\mu(D(z,r))(1-|z|^2)^{-3-ms-\frac{s\a}{p}}$.
\end{Lemma}

\begin{Lemma} \cite[Theorems 4.2 and 4.3]{ag1}  or  \cite[Corollary 3.4]{wz2}\label{th2}
Let $0<p,q,s<\infty$, $\a>-2$, $m\in\NN\cup\{0\}$ and let $\mu$ be a positive Borel measure on $\D$. Then the following conditions are equivalent:\\
(1) $D^{(m)}:AT_p^q(\a)\to L^s(\mu)$ is bounded;\\
(2) For any $r\in(0,1)$ and an $(r,\kappa)$-lattice $Z=\{a_j\}$, the sequence 
$$\hat{\mu}=\{\hat{\mu}_j\}=\left\{ \frac{\mu(D(a_j,r))}{(1-|a_j|^2)^{\frac{s(\a+2)}{p}+sm+1}}  \right\}
$$
satisfies one of the following conditions:
	
(a) If $q>s$ and $p>s$, then $\hat{\mu}\in T^{\frac{q}{q-s}}_{\frac{p}{p-s}}(Z)$;
	
(b) If $p\le s<q$, then $\hat{\mu}\in T^{\frac{q}{q-s}}_\infty(Z)$;
	
(c) If $q=s<p$, then $\hat{\mu}\in T^\infty_{\frac{p}{p-s}}(Z)$;
	
(d) If $p\le q=s$ or $q<s$, then $\left\{\hat{\mu}_j\cdot(1-|a_j|^2)^{1-\frac{s}{q}}\right\}\in l^\infty$.
\end{Lemma}

\section{The boundedness of $T_g^{n,k}:AT_p^q(\alpha)\to AT_t^s(\beta)$ }

The purpose of this section is to study  the boundedness  of the operator $T_g^{n,k}:AT_p^q(\a)\to AT_t^s(\b)$.
Now, we are in a position to state and prove the main result in this section.

\begin{Theorem}\label{th3}
Let $0<p,q,s,t<\infty$, $\a,\b>-2$, $g\in H(\D)$, $n\in\NN$  and $k\in\NN\cup\{0\}$ such that $0\le  k< n$ and $(n-k)pt+p(\b+2)-t(\a+2)>0$. 
Then $T_g^{n,k}:AT_p^q(\a)\to AT_t^s(\b)$ is bounded if and only if one of the following conditions holds:
	
(a) If $0<s<q<\infty$, $0<t<p<\infty$, 
$g^{(n-k)}\in AT_{\frac{pt}{p-t}}^{\frac{qs}{q-s}}(\delta),$ where $\delta=\left(n-k+\frac{\b}{t}-\frac{\a}{p} \right)\frac{pt}{p-t}$.
	
(b) If $0<s<q<\infty$, $0<p \leq t<\infty$, 
$
g^{(n-k)}(\cdot)(1-|\cdot|^2)^{n-k+\frac{\b+2}{t}-\frac{\a+2}{p}}\in AT_\infty^{\frac{qs}{q-s}}.
$
	
(c) If $0<q=s<\infty$, $0<t<p<\infty$, $g^{(n-k)}\in AT_{\frac{pt}{p-t}}^{\infty}(\delta),$ where $\delta=\left(n-k+\frac{\b}{t}-\frac{\a}{p} \right)\frac{pt}{p-t}$.
	
(d) If $0<q<s<\infty$, $0<p,t <\infty$ or $0<q=s<\infty$, $0<p\leq t<\infty$, 
$$\sup_{z\in \DD}|g^{(n-k)}(z)|(1-|z|^2)^{n-k+\frac{\b+2}{t}-\frac{\a+2}{p}+\frac{1}{s}-\frac{1}{q}}<\infty.$$
\end{Theorem}
\begin{proof}
{\bf Sufficiency.}
(a) If $0<s<q<\infty$, $0<t<p<\infty$, 
$g^{(n-k)}\in AT_{\frac{pt}{p-t}}^{\frac{qs}{q-s}}(\delta)$, where $\delta=\left( n-k+\frac{\b}{t}-\frac{\a}{p} \right)\frac{pt}{p-t}$.
Using Lemma \ref{Tfn} and H\"older's inequalities with $\frac{p}{t}>1$ and $\frac{q}{s}>1$ respectively, we obtain

\begin{equation}\label{aa}
\begin{aligned}
&\|T_g^{n,k}f\|_{AT_t^s(\b)}^s\\
\asymp&\int_{\mathbb{T}}\left( \int_{\Gamma(\eta)}|(T_g^{n,k}f)^{(n)}(z)|^t(1-|z|^2)^{nt+\b}dA(z)  \right)^{\frac{s}{t}}|d\eta|\\
=&\int_{\mathbb{T}}\left( \int_{\Gamma(\eta)}|f^{(k)}(z)|^t|g^{(n-k)}(z)|^t(1-|z|^2)^{nt+\b}dA(z)  \right)^{\frac{s}{t}}|d\eta|\\
\lesssim& \int_{\mathbb{T}} \left(   \int_{\Gamma(\eta)}|f^{(k)}(z)|^p(1-|z|^2)^{kp+\a}dA(z) \right)^{\frac{s}{p}} \cdot\left(\int_{\Gamma(\eta)}|g^{(n-k)}(z)|^{\frac{pt}{p-t}}(1-|z|^2)^{\delta} dA(z) \right)^{\frac{s(p-t)}{pt}}|d\eta| \\
\lesssim&\left(  \int_{\mathbb{T}} \left(   \int_{\Gamma(\eta)}|f^{(k)}(z)|^p(1-|z|^2)^{kp+\a}dA(z) \right)^{\frac{q}{p}} |d\eta|\right)^{\frac{s}{q}}\\
&\cdot\left(\int_{\mathbb{T}} \left(\int_{\Gamma(\eta)}|g^{(n-k)}(z)|^{\frac{pt}{p-t}}(1-|z|^2)^{\delta} dA(z) \right)^{\frac{qs(p-t)}{pt(q-s)}}|d\eta|  \right)^{\frac{q-s}{q}}\\
\lesssim& \|f\|_{AT_p^q(\a)}^s\cdot\|g^{(n-k)}\|^s_{AT_{\frac{pt}{p-t}}^{\frac{qs}{q-s}}(\delta)}.
\end{aligned}
\end{equation}
	
(b) If $0<s<q<\infty$, $0<p \leq t<\infty$, $
g^{(n-k)}(\cdot)(1-|\cdot|^2)^{n-k+\frac{\b+2}{t}-\frac{\a+2}{p}}\in AT_\infty^{\frac{qs}{q-s}}.$
Using Lemma \ref{Tfn}, H\"older's inequality and Lemma \ref{le1}, we get 
\begin{equation*} 
\begin{aligned}
 &\|T_g^{n,k}f\|_{AT_t^s(\b)}^s 
\asymp\int_{\mathbb{T}}\left( \int_{\Gamma(\eta)}|(T_g^{n,k}f)^{(n)}(z)|^t(1-|z|^2)^{nt+\b}dA(z)  \right)^{\frac{s}{t}}|d\eta|\\
=&\int_{\mathbb{T}}\left( \int_{\Gamma(\eta)}|f^{(k)}(z)|^t|g^{(n-k)}(z)|^t(1-|z|^2)^{nt+\b}dA(z)  \right)^{\frac{s}{t}}|d\eta|\\
\lesssim& \int_{\mathbb{T}} \left(   \int_{\Gamma(\eta)}|f^{(k)}(z)|^t(1-|z|^2)^{kt+\frac{t(\a+2)}{p}-2}dA(z) \right)^{\frac{s}{t}}\\
&\cdot\left(\sup_{z\in \Gamma(\eta)}|g^{(n-k)}(z)|^{t}(1-|z|^2)^{\b+2+nt-kt-\frac{t(\a+2)}{p}}  \right)^{\frac{s}{t}}|d\eta| \\
\lesssim& \left(  \int_{\mathbb{T}} \left(   \int_{\Gamma(\eta)}|f^{(k)}(z)|^t(1-|z|^2)^{kt+\frac{t(\a+2)}{p}-2}dA(z) \right)^{\frac{q}{t}} |d\eta|\right)^{\frac{s}{q}}\\
&\cdot\left(\int_{\mathbb{T}} \left(\sup_{z\in \Gamma(\eta)}|g^{(n-k)}(z)|(1-|z|^2)^{\frac{\b+2}{p}+n-k-\frac{(\a+2)}{p}}  \right)^{\frac{qs}{q-s}}|d\eta|  \right)^{\frac{q-s}{q}} 
\end{aligned}
\end{equation*}
\begin{equation}\label{bb}
\begin{aligned}
\lesssim& \|g^{(n-k)}(\cdot)(1-|\cdot|^2)^{n-k+\frac{\b+2}{t}-\frac{\a+2}{p}}\|^s_{AT_\infty^{\frac{qs}{q-s}}} \left(  \int_{\mathbb{T}} \left(   \int_{\Gamma(\eta)}|f^{(k)}(z)|^t(1-|z|^2)^{kt+\frac{t(\a+2)}{p}-2}dA(z) \right)^{\frac{q}{t}} |d\eta|\right)^{\frac{s}{q}}\\
\lesssim& \|g^{(n-k)}(\cdot)(1-|\cdot|^2)^{n-k+\frac{\b+2}{t}-\frac{\a+2}{p}}\|^s_{AT_\infty^{\frac{qs}{q-s}}}\|f\|^s_{AT_p^q(\a)}.
\end{aligned}
\end{equation}
	
(c) If $0<q=s<\infty$, $0<t<p<\infty$, 
$g^{(n-k)}\in AT_{\frac{pt}{p-t}}^{\infty}(\delta)$, where $\delta=\left(n-k+\frac{\b}{t}-\frac{\a}{p} \right)\frac{pt}{p-t}$.
Using Lemma \ref{Tfn} and  H\"older's inequality, we have
\begin{equation} \label{cc}
\begin{aligned}
\|T_g^{n,k}f\|_{AT_t^s(\b)}^q\asymp&\int_{\mathbb{T}}\left( \int_{\Gamma(\eta)}|(T_g^{n,k}f)^{(n)}(z)|^t(1-|z|^2)^{nt+\b}dA(z)  \right)^{\frac{q}{t}}|d\eta|\\
=&\int_{\mathbb{T}}\left( \int_{\Gamma(\eta)}|f^{(k)}(z)|^t|g^{(n-k)}(z)|^t(1-|z|^2)^{nt+\b}dA(z)  \right)^{\frac{s}{t}}|d\eta|\\
\lesssim& \int_{\mathbb{T}} \left(   \int_{\Gamma(\eta)}|f^{(k)}(z)|^p(1-|z|^2)^{kp+\a}dA(z) \right)^{\frac{q}{p}}\\
&\cdot\left(\int_{\Gamma(\eta)}|g^{(n-k)}(z)|^{\frac{pt}{p-t}}(1-|z|^2)^\delta dA(z) \right)^{\frac{q(p-t)}{pt}}|d\eta| \\
\lesssim&\|f\|^q_{AT_p^q(\a)}\sup_{\eta\in\mathbb{T}}\left( \int_{\Gamma(\eta)}|g^{n-k}(z)|^{\frac{pt}{p-t}}(1-|z|^2)^\delta dA(z) \right)^{\frac{q(p-t)}{pt}}
\\
\asymp&\|f\|^q_{AT_p^q(\a)} \sup_{\eta\in\mathbb{T}}\left(
\int_{\D}\chi_{\Gamma(\eta)}(z)|g^{n-k}(z)|^{\frac{pt}{p-t}}\frac{(1-|z|^2)^{\delta+1}}{|1-\bar{\eta}z|}dA(z)
\right)^{\frac{q(p-t)}{pt}}\\
\lesssim&\|f\|^q_{AT_p^q(\a)} \left(
\|g^{(n-k)}\|^{\frac{pt}{p-t}}_{AT_{\frac{pt}{p-t}}^{\infty}(\delta)}\cdot\sup_{0<\rho<1}\left\| \frac{\chi_\Gamma(\eta)(\cdot)}{|1-\bar{\eta}\cdot(\cdot)|} \right\|_{L^1(\rho\mathbb{T})}
\right)^{\frac{q(p-t)}{pt}}\\
\lesssim&\|f\|^q_{AT_p^q(\a)}
\|g^{(n-k)}\|^{q}_{AT_{\frac{pt}{p-t}}^{\infty}(\delta)}.	
\end{aligned}
\end{equation}
	
(d) If $0<q<s<\infty$, $0<p,t <\infty$ or $0<q=s<\infty$, $0<p\leq t<\infty$,
$\sup_{z\in \DD}|g^{(n-k)}(z)|(1-|z|^2)^{n-k+\frac{\b+2}{t}-\frac{\a+2}{p}+\frac{1}{s}-\frac{1}{q}}<\infty$.
Assume first that $0<q<s<\infty$, $0<p,t <\infty$. From \cite[Lemma 3.1]{wz1}, we know that
$$
\|f\|_{AT_t^s(\b)}\lesssim\|f'\|_{A^q_{\frac{q(2+\b+t)}{t}+\frac{q}{s}-2}}\asymp\|f^{(n)}\|_{A^q_{\frac{q(\b+2)}{t}+nq+\frac{q}{s}-2}}.
$$
Thus, when $q<s$, there exists $r$ such that $q<r<s$. Therefore, using Lemma \ref{le2}, we obtain
\begin{equation}\label{dd}
\begin{aligned}
\|T_g^{n,k}f\|_{AT_t^s(\b)}^s&\lesssim\|(T_g^{n,k}f)^{(n)}\|^s_{A^r_{\frac{r(\b+2)}{t}+nr+\frac{r}{s}-2}}\\
&\asymp\left( \int_\D|f^{(k)}(z)|^r|g^{(n-k)}(z)|^r(1-|z|^2)^{\frac{r(\b+2)}{t}+nr+\frac{r}{s}-2} dA(z) \right)^{\frac{s}{r}}\\
&\lesssim \left( \int_\D|g^{(n-k)}(z)|^r(1-|z|^2)^{\delta r}|f^{(k)}(z)|^r(1-|z|^2)^{\frac{r(\a+2)}{p}+kr+\frac{r}{q}-2} dA(z) \right)^{\frac{s}{r}}\\
&\lesssim\left( \sup_{z\in\D}|g^{(n-k)}(z)|(1-|z|^2)^\delta \right)^s \left( \int_\D |f^{(k)}(z)|^r(1-|z|^2)^{\frac{r(\a+2)}{p}+kr+\frac{r}{q}-2}dA(z) \right)^{\frac{s}{r}}\\
&\lesssim\|f\|^s_{AT_p^q(\a)}.
\end{aligned}
\end{equation}
Next, suppose that $0<q=s<\infty$, $0<p\leq t<\infty$. Using Lemmas \ref{Tfn} and   \ref{le1}, we have
\begin{equation}\label{ee}
\begin{aligned}
&\|T_g^{n,k}f\|_{AT_t^s(\b)}^q\asymp\int_{\mathbb{T}}\left( \int_{\Gamma(\eta)}|(T_g^{n,k}f)^{(n)}(z)|^t(1-|z|^2)^{nt+\b}dA(z)  \right)^{\frac{q}{t}}|d\eta|\\
=&\int_{\mathbb{T}}\left( \int_{\Gamma(\eta)}|f^{(k)}(z)|^t|g^{(n-k)}(z)|^t(1-|z|^2)^{nt+\b}dA(z)  \right)^{\frac{q}{t}}|d\eta|\\
\lesssim& \left(\sup_{z\in \Gamma(\eta)}|g^{(n-k)}(z)|^{t}(1-|z|^2)^{\delta}  \right)^{q}  \int_{\mathbb{T}} \left(   \int_{\Gamma(\eta)}|f^{(k)}(z)|^t(1-|z|^2)^{kt+\frac{t(\a+2)}{p}-2}dA(z) \right)^{\frac{s}{t}}|d\eta|\\
\lesssim& \|f\|^q_{AT_p^q(\a)},
\end{aligned}
\end{equation}
which implies the boundedness of $T_g^{n,k}$.
	
{\bf Necessity.}
Let  $Z=\{a_j\}$ be an $(r, \kappa)$-lattice and $\{x_j\}\in T_p^q(Z)$. Let $r_j:[0,1] \rightarrow\{-1,1\}$ be the Radermacher functions. For $L>\max \{1, \frac{p}{q}, \frac{1}{q}, \frac{1}{p}\}$, set
$$
F_v(z)=\sum_{j=1}^\infty x_j r_j(v) \frac{(1-|a_j|^2)^{L}}{(1-\overline{a_j} z)^{L+\frac{\a+2}{p}}}, \quad z \in \mathbb{D} .
$$
Using Proposition $\ref{2.23}$, we get $
\|F_v\|_{AT_p^q(\a)}\lesssim\|\{x_j\}\|_{T_p^q(Z)}.$
By the assumption, we get
\begin{align*} &\int_\mathbb{T}\left(\int_{\Gamma(\eta)}|(T_g^{n,k}F_v)^{(n)}(z)|^t(1-|z|^2)^{nt+\b} dA(z)\right)^{\frac{s}{t}}|d \eta|\\
\lesssim& \|T_g^{n,k}\|^s \|F_v\|_{AT_p^q(\alpha)}^s
\lesssim \|T_g^{n,k}\|^s \|\{x_j\}\|_{T_p^q(Z)}^s.
\end{align*}
Integrating both sides with respect to $v$ from $0$ to $1$,
using Fubini's theorem and Kahane's inequality (see \cite[Theorem 2.1]{ka}), we have
\begin{align*}
&\int_\mathbb{T}\left(\int_{0}^{1}\int_{\Gamma(\eta)}
|F_v^{(k)}(z)|^t|g^{(n-k)}(z)|^t
(1-|z|^2)^{nt+\b} dA(z) dv \right)^{\frac{s}{t}} |d \eta|\\
\lesssim&\int_\mathbb{T}\int_{0}^{1}\left(\int_{\Gamma(\eta)}|F_v^{(k)}(z)|^t|g^{(n-k)}(z)|^t(1-|z|^2)^{nt+\b} dA(z)\right)^{\frac{s}{t}} dv |d \eta|
\lesssim \|T_g^{n,k}\|^s \|\{x_j\}\|_{T_p^q(Z)}^s.
\end{align*}
Applying Khinchine's inequality, we obtain
\begin{align*}
\int_\mathbb{T}\left(\int_{\Gamma(\eta)}\left(\sum_j|x_j|^2 |a_j|^{2k}\frac{(1-|a_j|^2)^{2 L}}{|1-\overline{a_j} z|^{2 L+2k+\frac{2(\a+2)}{p}}}\right)^{\frac{t}{2}} |g^{n-k}(z)|^t
(1-|z|^2)^{nt+\b} dA(z)\right)^{\frac{s}{t}}|d \eta| 
\lesssim \|T_g^{n,k}\|^s\|\{x_j\}\|_{T_p^q(Z)}^s.
\end{align*}
Hence, 
\begin{align*}
&\int_\mathbb{T}\left(\int_{\Gamma(\eta)}\left(\sum_j|x_j|^2 \frac{(1-|a_j|^2)^{2 L}}{|1-\overline{a_j} z|^{2 L+2k+\frac{2(\a+2)}{p}}}\right)^{\frac{t}{2}} |g^{n-k}(z)|^t
(1-|z|^2)^{nt+\b} dA(z)\right)^{\frac{s}{t}}|d \eta| \\
\lesssim& \|T_g^{n,k}\|^s\|\{x_j\}\|_{T_p^q(Z)}^s.
\end{align*}
Let $w_j=|g^{(n-k)}(a_j)|(1-|a_j|^2)^{n-k+\frac{\b+2}{t}-\frac{\a+2}{p}}$. 
From Lemma \ref{2.9} we get  $\cup_{z\in\Gamma(\eta)}D(z,4r)\subset\widetilde{\Gamma}(\eta)$, and using Proposition \ref{2.7}, Fubini's theorem and H\"older inequality, we have
\begin{align*}
&\int_\mathbb{T}\left(\sum_{a_j \in \Gamma(\eta)}|x_j|^t |w_j|^t\right)^{\frac{s}{t}}|d \eta| \\ 
=&\int_\mathbb{T}\left(\sum_{a_j \in \Gamma(\eta)}|x_j|^t |g^{(n-k)}(a_j)|^t(1-|a_j|^2)^{nt-kt+\b+2-\frac{t(\a+2)}{p}}\right)^{\frac{s}{t}}|d \eta| \\ 
\lesssim&\int_\mathbb{T}\left(\sum_{a_j \in \Gamma(\eta)}|x_j|^t \int_{D(a_j, 4r)}|g^{(n-k)}(z)|^t \frac{(1-|z|^2)^{nt+\b}(1-|a_j|^2)^{tL}}
{|1-\ol{a_j}z|^{tL+tk+\frac{t(\a+2)}{p}}}dA(z)\right)^{\frac{s}{t}}|d \eta| \\
\lesssim&\int_{\mathbb{T}}\left(\int_{\widetilde{\Gamma}(\eta)}\sum_j|x_j|^t 
\frac{(1-|a_j|^2)^{tL}}{|1-\ol{a_j}z|^{tL+tk+\frac{t(\a+2)}{p}}}\chi_{D(a_j,4r)}(z)|g^{(n-k)}(z)|^t(1-|z|^2)^{nt+\b}dA(z)\right)^{\frac{s}{t}}|d\eta|\\
\lesssim&\int_{\mathbb{T}}\left(\int_{\widetilde{\Gamma}(\eta)}\left(\sum_j|x_j|^2 
\frac{(1-|a_j|^2)^{2L}}{|1-\ol{a_j}z|^{2L+2k+\frac{2(\a+2)}{p}}}\right)^\frac{t}{2}|g^{(n-k)}(z)|^t(1-|z|^2)^{nt+\b}dA(z)\right)^{\frac{s}{t}}|d \eta|\\
\lesssim&\|T_g^{n,k}\|^s \|\{x_j\}\|_{T_p^q(Z) }^s.
\end{align*}
Let $b>\max \left\{\frac{1}{s}, \frac{1}{t}\right\}$. For any $\{\iota_j\} \in T_{\frac{bt}{bt-1}}^{\frac{bs}{bs-1}}(Z)$,
using Fubini's theorem and H\"older's inequality twice, we obtain
\begin{equation}\label{4.5}
\begin{aligned}
&\sum_j |\iota_j x_j^{\frac{1}{b}} w_j^{\frac{1}{b}}|(1-|a_j|^2) 
\asymp \int_{\mathbb{T}}\left(\sum_{a_j \in \Gamma(\eta)} |\iota_j| |x_j|^{\frac{1}{b}} |w_j|^{\frac{1}{b}}\right)|d \eta| \\
\lesssim& \int_{\mathbb{T}}\left(\sum_{a_j \in \Gamma(\eta)}|x_j|^t  |w_j|^t \right)^{\frac{1}{tb}}\left(\sum_{a_j \in \Gamma(\eta)} |\iota_j|^{\frac{b t}{b t-1}}\right)^{\frac{b t-1}{b t}}|d \eta| \\
\leq&\left(\int_{\mathbb{T}}\left(\sum_{a_j \in \Gamma(\eta)}|x_j|^t |w_j|^t\right)^{\frac{sb}{tb}}|d \eta|\right)^{\frac{1}{bs}}
\left( \int_{\mathbb{T}} \left(  \sum_{a_j \in \Gamma(\eta)} |\iota_j|^{\frac{bt}{bt-1}} \right)^\frac{s(bt-1)}{t(bs-1)}   |d\eta| \right)^{\frac{bs-1}{bs}}\\
\lesssim&\|T_g^{n,k}\|^{\frac{1}{b}} \|\{x_j\}\|_{T_p^q(Z)}^{\frac{1}{b}}\|\{\iota_j\}\|_{T_{\frac{bt}{bt-1}}^{\frac{bs}{bs-1}}(Z)}.
\end{aligned}
\end{equation}
Applying Proposition \ref{2.16} we get
$$
\{\iota_j x_j^\frac{1}{b}\} \in T_{\frac{p b t}{pbt-p+t}}^{\frac{q b s}{qbs-q+s}}(Z)=T_{\frac{b t}{b t-1}}^{\frac{b s}{b s-1}}(Z) \cdot T_{b p}^{b q}(Z).
$$
Now, we divide the rest proof into four cases.

(a) For $0<s<q<\infty$, $0<t<p<\infty$, using Proposition \ref{2.18}, we have $$\left(T_{\frac{p b t}{pbt-p+t}}^{\frac{q b s}{qbs-q+s}}(Z)\right)^*=T_{\frac{pb t}{p-t}}^{\frac{ qb s}{q-s}}(Z),$$ 
which together with $(\ref{4.5})$ imply
\begin{align*}
&\left\|\left\{ |w_j|^{\frac{1}{b}}\right\}\right\|_{T_{\frac{pb t}{p-t}}^{\frac{ qb s}{q-s}}(Z)}=\left\|\left\{|w_j|\right\}\right\|_{T_{\frac{p t}{p-t}}^{\frac{q s}{q-s}}(Z)}^\frac{1}{b}\\ =&\left(\int_{\mathbb{T}}\left(\sum_{a_j \in \Gamma(\eta)}\left( |g^{(n-k)}(a_j)|(1-|a_j|^2)^{n-k+\frac{\b+2}{t}-\frac{\a+2}{p}} \right)^{\frac{p t}{p-t}}\right)^{\frac{(p-t) q s}{(q-s) p t}}|d \eta|\right)^{\frac{q-s}{qbs}}<\infty .
\end{align*} 
Thus, we have
$|g^{(n-k)}(a_j)|(1-|a_j|^2)^{n-k+\frac{\b+2}{t}-\frac{\a+2}{p}}\in T_{\frac{pt}{p-t}}^{\frac{qs}{q-s}}(Z)$.
By Lemma \ref{th2}, we obtain for $b>\max\{\frac{1}{t},\frac{1}{s}\}$ and $m\in\NN\cup\{0\}$, the operator
$D^{(m)}:AT_{\frac{p b t}{pbt-p+t}}^{\frac{q b s}{qbs-q+s}}(\gamma)\to L^1(\mu_g)$
is bounded.
Here, $\gamma=\left[(\b+t)\left(1-\frac{1}{tb}\right)+\frac{\a}{pb}\right]\cdot\frac{pbt}{pbt-p+t}$ and $d\mu_g(z)=|g^{(n-k)}(z)|^\frac{1}{b}(1-|z|^2)^{m+\b+t+1+\frac{n-k-1}{b}}dA(z)$.
According to Lemma \ref{th1}, we know that the operator $D^{(m)}:AT_{\frac{p b t}{pbt-p+t}}^{\frac{q b s}{qbs-q+s}}(\gamma)\to L^1(\mu_g)$ is bounded if and only if 
$g^{(n-k)}\in AT_{\frac{pt}{p-t}}^{\frac{qs}{q-s}}(\delta)$, where $\delta=\left( n-k+\frac{\b}{t}-\frac{\a}{p} \right)\frac{pt}{p-t}$.

(b) Let $0<s<q<\infty$.
By Proposition \ref{2.18}  for  $0<p=t<\infty$ and Proposition \ref{2.19}  for  $0<p<t<\infty$, 
we obtain 
$$\left(T_{\frac{p b t}{pbt-p+t}}^{\frac{q b s}{qbs-q+s}}(Z)\right)^*=T_{\frac{pb t}{p-t}}^{\frac{ qb s}{q-s}}(Z)\quad  \mbox{ and} \quad \left(T_{\frac{p b t}{pbt-p+t}}^{\frac{q b s}{qbs-q+s}}(Z)\right)^*=T_{\infty}^{\frac{ qb s}{q-s}}(Z),$$ 
respectively. By $(\ref{4.5})$ we have
\begin{align*}
&\left\|\left\{|w_j|^{\frac{1}{b}}\right\}\right\|_{T_{\infty}^{\frac{ qb s}{q-s}}(Z)}=\left\|\left\{|w_j|\right\}\right\|_{T_{\infty}^{\frac{q s}{q-s}}(Z)}^\frac{1}{b}\\
=&\left(\int_{\mathbb{T}}\left(\sup _{a_j \in \Gamma(\eta)} |g^{(n-k)}(a_j)|(1-|a_j|^2)^{n-k+\frac{\b+2}{t}-\frac{\a+2}{p}}\right)^{\frac{qs}{q-s}}|d \eta|\right)^{\frac{q-s}{qbs}}<\infty .
\end{align*}
Thus, we have $|g^{(n-k)}(a_j)|(1-|a_j|^2)^{n-k+\frac{\b+2}{t}-\frac{\a+2}{p}}\in T_\infty^{\frac{qs}{q-s}}(Z)$.
By Lemma \ref{th2}, we obtain for $b>\max\{\frac{1}{t},\frac{1}{s}\}$ and $m\in\NN\cup\{0\}$, the operator
$D^{(m)}:AT_{\frac{p b t}{pbt-p+t}}^{\frac{q b s}{qbs-q+s}}(\gamma)\to L^1(\mu_g)$
is bounded.  
Here, $\gamma=\left[(\b+t)\left(1-\frac{1}{tb}\right)+\frac{\a}{pb}\right]\cdot\frac{pbt}{pbt-p+t}$ and $d\mu_g(z)=|g^{(n-k)}(z)|^\frac{1}{b}(1-|z|^2)^{m+\b+t+1+\frac{n-k-1}{b}}dA(z)$.
Using Lemma \ref{th1}, we have that $D^{(m)}:AT_{\frac{p b t}{pbt-p+t}}^{\frac{q b s}{qbs-q+s}}(\gamma)\to L^1(\mu_g)$ is bounded if and only if 
%
$g^{(n-k)}(\cdot)(1-|\cdot|^2)^{n-k+\frac{\b+2}{t}-\frac{\a+2}{p}}\in AT_\infty^{\frac{qs}{q-s}}$.

(c) For $0<q=s<\infty$, $0<t<p<\infty$, applying  Proposition $\ref{2.20}$, we get $$\left(T_{\frac{p b t}{pbt-p+t}}^{1}(Z)\right)^*=T_{\frac{pb t}{p-t}}^{\infty}(Z).$$ 
Using $(\ref{4.5})$ we have
{\small\begin{align*}
&\left\|\left\{|w_j|^{\frac{1}{b}}\right\}\right\|_{T_{\frac{ pb t}{p-t}}^{\infty}(Z)}=\left\|\left\{|w_j|\right\}\right\|_{T_{\frac{p t}{p-t}}^{\infty}(Z)}^\frac{1}{b}\\
=&\operatorname{esssup} _{\eta \in \mathbb{T}}\left(\sup _{u\in\Gamma(\eta)}\frac{1}{1-|u|^2} \sum_{a_j \in S(u)}\left( |g^{(n-k)}(a_j)|(1-|a_j|^2)^{n-k+\frac{\b+2}{t}-\frac{\a+2}{p}} \right)^{\frac{pt}{p-t}}(1-|a_j|^2)\right)^{\frac{p-t}{pbt}}<\infty.
\end{align*}}
Thus, we have $|g^{(n-k)}(a_j)|(1-|a_j|^2)^{n-k+\frac{\b+2}{t}-\frac{\a+2}{p}}\in T^\infty_{\frac{pt}{p-t}}(Z)$.
By Lemma \ref{th2}, we obtain for $b>\max\{\frac{1}{t},\frac{1}{s}\}$ and $m\in\NN\cup\{0\}$, the operator
$D^{(m)}:AT_{\frac{p b t}{pbt-p+t}}^{\frac{q b s}{qbs-q+s}}(\gamma)\to L^1(\mu_g)$
is bounded. 
Here, $\gamma=\left[(\b+t)\left(1-\frac{1}{tb}\right)+\frac{\a}{pb}\right]\cdot\frac{pbt}{pbt-p+t}$ and $d\mu_g(z)=|g^{(n-k)}(z)|^\frac{1}{b}(1-|z|^2)^{m+\b+t+1+\frac{n-k-1}{b}}dA(z)$.
From Lemma \ref{th1}, we know that $D^{(m)}:AT_{\frac{p b t}{pbt-p+t}}^{\frac{q b s}{qbs-q+s}}(\gamma)\to L^1(\mu_g)$ is bounded if and only if $g^{(n-k)}\in AT_{\frac{pt}{p-t}}^{\infty}(\delta)$, where $\delta=\left(n-k+\frac{\b}{t}-\frac{\a}{p} \right)\frac{pt}{p-t}$.

(d) For $0<q<s<\infty$, $0<p,t<\infty$ or $0<q=s<\infty$, $0<p\le t<\infty$, using Proposition \ref{2.21} and $(\ref{4.5})$, we have
\begin{align*}
\sup _j |w_j|^{\frac{1}{b}}(1-|a_j|^2)^{\frac{q-s}{q b s}}=\left(\sup _j |g^{(n-k)}(a_j)|(1-|a_j|^2)^{n-k+\frac{\b+2}{t}-\frac{\a+2}{p}+\frac{1}{s}-\frac{1}{q}}\right)^\frac{1}{b}<\infty.
\end{align*}
Thus, we have $ |g^{(n-k)}(a_j)|(1-|a_j|^2)^{n-k+\frac{\b+2}{t}-\frac{\a+2}{p}+\frac{1}{s}-\frac{1}{q}}\in l^\infty$.  
By Lemma \ref{th2}, we obtain for $b>\max\{\frac{1}{t},\frac{1}{s}\}$ and $m\in\NN\cup\{0\}$, the operator
$D^{(m)}:AT_{\frac{p b t}{pbt-p+t}}^{\frac{q b s}{qbs-q+s}}(\gamma)\to L^1(\mu_g)$
is bounded.
Here, $\gamma=\left[(\b+t)\left(1-\frac{1}{tb}\right)+\frac{\a}{pb}\right]\cdot\frac{pbt}{pbt-p+t}$ and $d\mu_g(z)=|g^{(n-k)}(z)|^\frac{1}{b}(1-|z|^2)^{m+\b+t+1+\frac{n-k-1}{b}}dA(z)$.
By Lemma \ref{th1},
the operator $D^{(m)}:AT_{\frac{p b t}{pbt-p+t}}^{\frac{q b s}{qbs-q+s}}(\gamma)\to L^1(\mu_g)$ is bounded if and only if
$$\sup_{z\in \DD}|g^{(n-k)}(z)|(1-|z|^2)^{n-k+\frac{\b+2}{t}-\frac{\a+2}{p}+\frac{1}{s}-\frac{1}{q}}<\infty.
$$
%
\end{proof}

Next, we consider the boundedness of the operator $T_g^{n,k}:AT_p^q(\a)\to AT_t^s(\b)$ when  $t=\infty$, i.e., $T_g^{n,k}:AT_p^q(\a)\to H^s$, since $AT_\infty^s=H^s$.

\begin{Theorem}\label{hb}
Let $0<p,q,s<\infty$, $\a>-2$, $g\in H(\D)$, $n\in\NN$  and $k\in\NN\cup\{0\}$ such that $0\le k<n$ and $(n-k)p-(\a+2)>0$. 
Then $T_g^{n,k}:AT_p^q(\a)\to H^s$ is bounded if and only if one of the following conditions holds:
	
(a) If $0<s<q<\infty$, $2<p<\infty$, 
	$g^{(n-k)}\in AT_{\frac{2p}{p-2}}^{\frac{qs}{q-s}}(\delta),$ where $\delta=\left(n-k -1-\frac{\a}{p} \right)\frac{2p}{p-2}$.
	
(b) If $0<s<q<\infty$, $0<p \leq 2$, 
	$	g^{(n-k)}(\cdot)(1-|\cdot|^2)^{n-k-\frac{\a+2}{p}}\in AT_\infty^{\frac{qs}{q-s}}.$
	
(c) If $0<q=s<\infty$, $2<p<\infty$, $g^{(n-k)}\in AT_{\frac{2p}{p-2}}^{\infty}(\delta),$ where $\delta=\left(n-k-1-\frac{\a}{p} \right)\frac{2p}{p-2}$.
	
(d) If $0<q<s<\infty$, $0<p<\infty$ or $0<q=s<\infty$, $0<p\leq 2$, 
$$\sup_{z\in \DD}|g^{(n-k)}(z)|(1-|z|^2)^{n-k-\frac{\a+2}{p}+\frac{1}{s}-\frac{1}{q}}<\infty.$$
\end{Theorem}

\begin{proof}
From \cite[Theorem G]{pa}, we have
\begin{align*}
\|T_g^{n,k}f\|^s_{H^s}\asymp\int_{\TT}\left(\int_{\Gamma(\eta)} |f^{(k)}(z)|^2|g^{(n-k)}(z)|^2(1-|z|^2)^{2n-2}dA(z) \right)^{\frac{s}{2}}|d\eta|.
\end{align*}
The proof is similar to the proof of Theorem \ref{th3}, so we omit the details here.
\end{proof}

\begin{Remark}
Using Theorem \ref{hb}, we can obtain the characterization of the boundedness of the operator $T_g^{n,k}:H^p\to H^s$, which has already been investigated by Chalmoukis in \cite{ch}.
In particular, we show that
$T_g^{n,k}:H^p\to H^s$ is bounded  if and only if $\sup\limits_{z\in\D}(1-|z|^2)^{\frac{1}{s}-\frac{1}{p}+n-k}|g^{(n-k)}(z)|<\infty$ when $0<p<s<\infty$. 
In addition, when $n=1$ and $k=0$, by Theorem \ref{hb} and the fact that
$A^p_\a=AT_p^p(\a-1)$,  we can get the characterization of the boundedness of the operator $T_g:A^p_\a\to H^s$, which has already been studied by Wu in \cite{w2} and subsequently by Miihkinen, Pau, Per\"al\"a, and Wang in \cite{mppw} in the case of the unit ball.
\end{Remark}

\section{The compactness of $T_g^{n,k}:AT_p^q(\alpha)\to AT_t^s(\beta)$}
In this section, we will study the compactness  of the operator $T_g^{n,k}:AT_p^q(\a)\to AT_t^s(\b)$.

\begin{Theorem}\label{th5}
Let $0<p,q,s,t<\infty$, $\a,\b>-2$, $g\in H(\D)$, $n\in\NN$ and $k\in\NN\cup\{0\}$ such that $0\le  k< n$ and $(n-k)pt+p(\b+2)-t(\a+2)>0$.  
Then $T_g^{n,k}:AT_p^q(\a)\to AT_t^s(\b)$ is compact if and only if $T_g^{n,k}:AT_p^q(\a)\to AT_t^s(\b)$ is bounded and one of the following conditions holds:
	
%
%
%
	
(a) If $0<s<q<\infty$, $0<t<p<\infty$, 
$g^{(n-k)}\in AT_{\frac{pt}{p-t}}^{\frac{qs}{q-s}}(\delta),$ where $\delta=\left(n-k+\frac{\b}{t}-\frac{\a}{p} \right)\frac{pt}{p-t}$.\\
	
(b) If $0<s<q<\infty$, $0<p \leq t<\infty$, 
$$
\lim_{R\to 1^{-}}\int_{\TT}\left( \sup_{{z}\in\Gamma(\eta)\backslash \ol{D(0,R)}} |g^{(n-k)}(z)|(1-|z|^2)^{n-k+\frac{\b+2}{t}-\frac{\a+2}{p}} \right)^{\frac{qs}{q-s}}|d\eta|=0.
$$
	
(c) If $0<q=s<\infty$, $0<t<p<\infty$, 
$$
\lim_{|u|\to 1^{-}}\frac{1}{1-|u|^2}\int_{ S(u)}|g^{(n-k)}(z)|^{\frac{pt}{p-t}}(1-|z|^2)^{\left(n-k+\frac{\b}{t}-\frac{\a}{p}\right){\frac{pt}{p-t}}+1}dA(z)=0.
$$
	
(d) If $0<q<s<\infty$, $0<p,t <\infty$ or $0<q=s<\infty$, $0<p\leq t<\infty$, 
$$\lim_{z\to1^-}|g^{(n-k)}(z)|(1-|z|^2)^{n-k+\frac{\b+2}{t}-\frac{\a+2}{p}+\frac{1}{s}-\frac{1}{q}}=0.$$
\end{Theorem}

\begin{proof}
{\bf Sufficiency.}
Consider the sequence $\{f_j\}_{j=1}^\infty\subset AT_p^q(\a)$ with the property that $\sup_j\|f_j\|_{AT_p^q(\a)}<\infty$. This implies that the set $\{f_j\}$ is uniformly bounded on compact subsets of the unit disc $\D$, and by Montel's theorem, it constitutes a normal family. Consequently, a uniformly convergent subsequence $\{f_{n_j}\}_{j=1}^\infty$ can be selected on compact subsets of $\D$ to an analytic function $f$. According to Fatou's Lemma, $f\in AT_p^q(\a)$. Let $h_j = f_{n_j} - f$, which also belongs to $AT_p^q(\a)$. The goal is to prove that $\lim_{j\to \infty}\|T_g^{n,k}h_j\|_{AT_t^s(\b)}=0$, which would establish that the operator $T_g^{n,k}:AT_p^q(\a)\to AT_t^s(\b)$ is compact.

(a) If $0<s<q<\infty$, $0<t<p<\infty$, 
%
$g^{(n-k)}\in AT_{\frac{pt}{p-t}}^{\frac{qs}{q-s}}(\delta)$, where $\delta=\left(n-k+ \frac{\b}{t}-\frac{\a}{p} \right)\frac{pt}{p-t}$. Then, by the dominated convergence theorem, for any $\epsilon>0$, there exists $R_0\in(0,1)$ such that
$$
\sup_{R\ge R_0}\left(\int_{\TT}\left( \int_{\Gamma(\eta)\backslash\ol{D(0,R_0)}} |g^{(n-k)}(z)|^{\frac{pt}{p-t}}(1-|z|^2)^{\delta}dA(z) \right)^{\frac{(p-t)qs}{pt(q-s)}}|d\eta|\right)^{\frac{q-s}{qs}}<\epsilon.
$$
Noting that $|h_j(z)|\to 0$ uniformly on any compact subsets of $\D$, we can choose $j_0$ large enough such that $\sup_{j\ge j_0,|z|\le R_0}|h_j^{(k)}(z)|<\epsilon$. Using $(\ref{aa})$, when $j\ge j_0$, we have
\begin{align*}
\|T_g^{n,k}h_j\|_{AT_t^s(\b)}^s\lesssim&\int_{\mathbb{T}}\left( \int_{\Gamma(\eta)\cap\{|z|\le R_0\}}|h_j^{(k)}(z)|^t|g^{(n-k)}(z)|^t(1-|z|^2)^{nt+\b}dA(z)  \right)^{\frac{s}{t}}|d\eta|\\
&+\int_{\mathbb{T}}\left( \int_{\Gamma(\eta)\backslash\ol{D(0,R_0)}}|h_j^{(k)}(z)|^t|g^{(n-k)}(z)|^t(1-|z|^2)^{nt+\b}dA(z)  \right)^{\frac{s}{t}}|d\eta|\\
\lesssim&\epsilon^s+\|h_j\|_{AT_p^q(\a)}^s\left(\int_{\mathbb{T}} \left(\int_{\Gamma(\eta)\backslash\ol{D(0,R_0)}}|g^{(n-k)}(z)|^{\frac{pt}{p-t}}(1-|z|^2)^\delta dA(z) \right)^{\frac{(p-t)qs}{pt(q-s)}}|d\eta|  \right)^{\frac{q-s}{q}}\\
\lesssim& \epsilon^s.
\end{align*}

(b) If $0<s<q<\infty$, $0<p \leq t<\infty$, the assumption
implies that for any $\epsilon>0$, there exists $R_0\in(0,1)$ such that
$$
\sup_{R\ge R_0}\left(\int_{\TT}\left( \sup_{z\in\Gamma(\eta)\backslash \ol{D(0,R)}} |g^{(n-k)}(z)|(1-|z|^2)^{n-k+\frac{\b+2}{t}-\frac{\a+2}{p}} \right)^{\frac{qs}{q-s}}|d\eta|\right)^{\frac{q-s}{qs}}<\epsilon.
$$
Choose $j_0$ such that $\sup_{j\ge j_0,|z|\le R_0}|h_j^{(k)}(z)|<\epsilon$. Using $(\ref{bb})$, when $j\ge j_0$, we get
\begin{align*}
\|T_g^{n,k}h_j\|_{AT_t^s(\b)}^s\lesssim&\int_{\mathbb{T}}\left( \int_{\Gamma(\eta)\cap\{|z|\le R_0\}}|h_j^{(k)}(z)|^t|g^{(n-k)}(z)|^t(1-|z|^2)^{nt+\b}dA(z)  \right)^{\frac{s}{t}}|d\eta|\\
&+\int_{\mathbb{T}}\left( \int_{\Gamma(\eta)\backslash\ol{D(0,R_0)}}|h_j^{(k)}(z)|^t|g^{(n-k)}(z)|^t(1-|z|^2)^{nt+\b}dA(z)  \right)^{\frac{s}{t}}|d\eta|\\
\lesssim&\epsilon^s+ \|h_j\|_{AT_p^q(\a)}^s  \left(\int_{\TT}  \left(\sup_{z\in \Gamma(\eta)\backslash\ol{D(0,R_0)}}|g^{(n-k)}(z)|(1-|z|^2)^{n-k+\frac{\b+2}{t}-\frac{\a+2}{p}}  \right)^{\frac{qs}{q-s}}|d\eta| \right)^{\frac{q-s}{q}} \\
\lesssim&\epsilon^s. 
\end{align*}

(c) If $0<q=s<\infty$, $0<t<p<\infty$, 
the assumption implies that
$$
\lim_{R\to 1^{-}}\sup_{u\in\D}\frac{1}{1-|u|^2}\int_{S(u)\backslash\ol{D(0,R)}}|g^{(n-k)}(z)|^{\frac{pt}{p-t}}(1-|z|^2)^{\left( n-k+\frac{\b}{t}-\frac{\a}{p} \right)\frac{pt}{p-t}+1}dA(z)=0.
$$
Hence, for any $\epsilon>0$, there exists $R_0\in(0,1)$ such that
$$
\sup_{u\in\D,R\ge R_0}\frac{1}{1-|u|^2}\int_{S(u)\backslash\ol{D(0,R)}}|g^{(n-k)}(z)|^{\frac{pt}{p-t}}(1-|z|^2)^{\left( n-k+\frac{\b}{t}-\frac{\a}{p} \right)\frac{pt}{p-t}+1}dA(z)<\epsilon.
$$
Choose $j_0$ such that $\sup_{j\ge j_0,|z|\le R_0}|h_j^{(k)}(z)|<\epsilon$. Using $(\ref{cc})$, when $j\ge j_0$, we obtain
\begin{align*}
 \|T_g^{n,k}h_j\|_{AT_t^s(\b)}^s 
\asymp&\int_{\mathbb{T}}\left( \int_{\Gamma(\eta)\cap\{|z|\le R_0\}}|h_j^{(k)}(z)|^t|g^{(n-k)}(z)|^t(1-|z|^2)^{nt+\b}dA(z)  \right)^{\frac{s}{t}}|d\eta|\\
&+\int_{\mathbb{T}}\left( \int_{\Gamma(\eta)\backslash\ol{D(0,R_0)}}|h_j^{(k)}(z)|^t|g^{(n-k)}(z)|^t(1-|z|^2)^{nt+\b}dA(z)  \right)^{\frac{s}{t}}|d\eta|\\
\lesssim&\epsilon^s+ \|h_j\|_{AT_p^q(\a)}^s\\
&\cdot
\left( \sup_{u\in\D,R\ge R_0}\frac{1}{1-|u|^2}\int_{S(u)\backslash\ol{D(0,R)}}|g^{(n-k)}(z)|^{\frac{pt}{p-t}}(1-|z|^2)^{\left( n-k+\frac{\b}{t}-\frac{\a}{p} \right)\frac{pt}{p-t}+1}dA(z)    \right)^\frac{s(p-t)}{pt}      \\
\lesssim&\epsilon^s. 
\end{align*} 

(d) If $0<q<s<\infty$, $0<p,t <\infty$ or $0<q=s<\infty$, $0<p\leq t<\infty$, the assumption implies that 
for any $\epsilon>0$, there exists $R_0\in(0,1)$ such that
$|g^{(n-k)}(z)|(1-|z|^2)^{n-k+\frac{\b+2}{t}-\frac{\a+2}{p}+\frac{1}{s}-\frac{1}{q}}<\epsilon$ for any $|z|\ge R_0$.
Choose $j_0$ such that $\sup_{j\ge j_0,|z|\le R_0}|h_j^{(k)}(z)|<\epsilon$.
Assume first that $0<q<s<\infty$, $0<p,t <\infty$, there exists $r$ such that $q<r<s$,
by $(\ref{dd})$ and Lemma \ref{le2}, when $j\ge j_0$, we get
\begin{equation*}
\begin{aligned}
\|T_g^{n,k}h_j\|_{AT_t^s(\b)}^s\lesssim&\|(T_g^{n,k}h_j)^{(n)}\|^s_{A^r_{\frac{r(\b+2)}{t}+nr+\frac{r}{s}-2}}\\
\asymp&\left( \int_\D|h_j^{(k)}(z)|^r|g^{(n-k)}(z)|^r(1-|z|^2)^{\frac{r(\b+2)}{t}+nr+\frac{r}{s}-2} dA(z) \right)^{\frac{s}{r}}\\
\lesssim&\left( 
\int_{|z|\le R_0}|h_j^{(k)}(z)|^r|g^{(n-k)}(z)|^r(1-|z|^2)^{\frac{r(\b+2)}{t}+nr+\frac{r}{s}-2} dA(z)
\right)^{\frac{s}{r}}\\
&+\left( 
\int_{R_0<|z|<1}|h_j^{(k)}(z)|^r|g^{(n-k)}(z)|^r(1-|z|^2)^{\frac{r(\b+2)}{t}+nr+\frac{r}{s}-2} dA(z)
\right)^{\frac{s}{r}}\\
\lesssim&\epsilon^s+\epsilon^s\left(\int_\DD |h_j^{(k)}(z)|^r(1-|z|^2)^{\frac{r(\a+2)}{p}+kr+\frac{r}{q}-2}dA(z)\right)^{\frac{s}{r}}\\
\lesssim&\epsilon^s. 
\end{aligned}
\end{equation*}
Next, suppose that $0<q=s<\infty$, $0<p\leq t<\infty$,
by $(\ref{ee})$ and Lemma \ref{le1}, when $j\ge j_0$, we get
\begin{align*}
\|T_g^{n,k}h_j\|_{AT_t^s(\b)}^s\lesssim&\int_{\mathbb{T}}\left( \int_{\Gamma(\eta)\cap\{|z|\le R_0\}}|h_j^{(k)}(z)|^t|g^{(n-k)}(z)|^t(1-|z|^2)^{nt+\b}dA(z)  \right)^{\frac{s}{t}}|d\eta|\\
&+\int_{\mathbb{T}}\left( \int_{\Gamma(\eta)\backslash\{|z|\le R_0\}}|h_j^{(k)}(z)|^t|g^{(n-k)}(z)|^t(1-|z|^2)^{nt+\b}dA(z)  \right)^{\frac{s}{t}}|d\eta|\\
\lesssim&\epsilon^s+\epsilon^s\int_{\mathbb{T}} \left(\int_{\Gamma(\eta)}|h_j^{(k)}(z)|^t(1-|z|^2)^{kt+\frac{t(\a+2)}{p}-2}dA(z) \right)^{\frac{s}{t}}|d\eta|  \\
\lesssim&\epsilon^s. 
\end{align*}
Thus, $T_g^{n,k}:AT_p^q(\a)\to AT_t^s(\b)$ is compact.

{\bf Necessity.}  Assume that $T_g^{n,k}:AT_p^q(\a)\to AT_t^s(\b)$ is compact.  It is evident that $T_g^{n,k}:AT_p^q(\a)\to AT_t^s(\b)$ is bounded and $(a)$ is true according to Theorem \ref{th3}. Thus, we only need to prove that $(b)$, $(c)$, and $(d)$ hold.
Let $Z=\{a_j\}$ be an $(r, \kappa)$-lattice and $E=\left\{\{x_j\}\in T_p^q(Z):\|\{\{x_j\}\|_{T_p^q(Z)}=1  \right\}$.  For $L>\max \{1, \frac{p}{q}, \frac{1}{q}, \frac{1}{p}\}$, set
$$
F(\{x_j\})(z)=\sum_{j=1}^\infty x_j  \frac{(1-|a_j|^2)^{L}}{(1-\overline{a_j} z)^{L+\frac{\a+2}{p}}}, \quad z \in \mathbb{D}.
$$
Using Proposition \ref{2.23}, we get $\|F(\{x_j\})\|_{T_p^q(\a)}\lesssim \|\{x_j\}\|_{T_p^q(Z)}$.
Since $T_g^{n,k}$ is compact and $F(\{x_j\})$ is a bounded set, the set $\{T_g^{n,k}\circ F(\{x_j\}):x_j\in E \}$ is relatively compact 
and hence a totally bounded set in $T_t^s(\b)$. Thus, for any $\epsilon>0$,
there exist a finite number of functions $h_1,\cdots,h_N$, such that $T_g^{n,k}\circ F\subset \bigcup_{i=1}^{N}D\left(h_i,\frac{\epsilon}{2}\right)$, where $$D\left(h,\frac{\epsilon}{2}\right)=\left\{ f\in T_g^{n,k}\circ F:\|f-h\|_{T_t^s(\b)}<\frac{\epsilon}{2} \right\}.$$
Noting that $\sup_{i=1,\cdots,N}\|h_i^{(n)}\|_{T_t^s(\b)}<\infty$, for the above $\epsilon>0$, there exists $R_0\in(0,1)$ such that
$$
\sup_{i=1,\cdots,N}\left( \int_{\TT}\left(   \int_{\Gamma(\eta)\backslash\ol{D(0,R)}} |h_i^{(n)}(z)|^t(1-|z|^2)^{nt+\b}dA(z)   \right)^{\frac{s}{t}} |d\eta|\right)^{\frac{1}{s}}<\frac{\epsilon}{2}
$$ 
for all $R> R_0$. Thus, for any $
\{x_j\}\in E$, there exists some $i_0\in\{1,\cdots,N\}$ such that $T_g^{n,k}\circ F\in D\left(h_{i_0},\frac{\epsilon}{2}\right)$. Therefore,
\begin{align*}
&\left( \int_{\TT}\left(   \int_{\Gamma(\eta)\backslash\ol{D(0,R)}} |(T_g^{n,k}F)^{(n)}(z)|^t(1-|z|^2)^{nt+\b}dA(z)  \right)^{\frac{s}{t}} |d\eta|\right)^{\frac{1}{s}}\\
=&\left( \int_{\TT}\left(   \int_{\Gamma(\eta)\backslash\ol{D(0,R)}} |F^{(k)}(z) g^{(n-k)}(z)|^t(1-|z|^2)^{nt+\b}dA(z)  \right)^{\frac{s}{t}} |d\eta|\right)^{\frac{1}{s}}\\
\lesssim& \left( \int_{\TT}\left(   \int_{\Gamma(\eta)\backslash\ol{D(0,R)}} |F^{(k)}(z)g^{(n-k)}(z)-h_{i_0}^{(n)}(z)|^t(1-|z|^2)^{nt+\b}dA(z)  \right)^{\frac{s}{t}} |d\eta|\right)^{\frac{1}{s}}\\
&+\left( \int_{\TT}\left(   \int_{\Gamma(\eta)\backslash\ol{D(0,R)}} |h_{i_0}^{(n)}(z)|^t(1-|z|^2)^{nt+\b}dA(z)   \right)^{\frac{s}{t}} |d\eta|\right)^{\frac{1}{s}}\\
\lesssim&\|T_g^{n,k}\circ F-h_{i_0}\|_{T_t^s(\b)}+\frac{\epsilon}{2}<\epsilon
\end{align*}
for all $R> R_0$, that is equivalent to
\begin{align*}
&\int_\mathbb{T}\left(\int_{\Gamma(\eta)\backslash\ol{D(0,R)}}\left(\sum_j|x_j|^2 \frac{(1-|a_j|^2)^{2 L}}{|1-\overline{a_j} z|^{2 L+2k+\frac{2(\a+2)}{p}}}\right)^{\frac{t}{2}} |g^{n-k}(z)|^t
(1-|z|^2)^{nt+\b} dA(z)\right)^{\frac{s}{t}}|d \eta| \\
\lesssim& \epsilon^s\|\{x_j\}\|_{T_p^q(Z)}^s.
\end{align*}
for any $\{x_j\}\in T_p^q(Z)$ and $R> R_0$. 
Let $r_j(v)$ be the Rademacher functions. Replacing $x_j$ by $x_j r_j(v)$, and by applying the same technique as in Theorem \ref{th3}, we obtain
\begin{align*}
\int_\mathbb{T}\left(\sum_{a_j \in \Gamma(\eta)}|x_j|^t |g^{(n-k)}(a_j)|^t(1-|a_j|^2)^{nt-kt+\b+2-\frac{t(\a+2)}{p}}\cdot\chi_{\{|z|\ge R\}}\right)^{\frac{s}{t}}|d \eta|
\lesssim \epsilon^s \|\{x_j\}\|_{T_p^q(Z)}^s
\end{align*}
for $R>R_0'=\inf\left\{|a_j|:D(a_j,\rho)\subset\{|z|\ge R_0\}\right\}$, 
where $\chi_{\{|z|\ge R\}}$ is the characteristic function.
	
(b)  If $0<s<q<\infty$, $0<p \leq t<\infty$, using the principles of duality and factorization for sequence tent spaces, as outlined in Theorem \ref{th3}, we deduce
\begin{align*}
\int_{\TT} \sup_{{a_j}\in\Gamma(\eta)\backslash \ol{D(0,R)}} |g^{(n-k)}(a_j)|^{\frac{qs}{q-s}}(1-|a_j|^2)^{(n-k+\frac{\b+2}{t}-\frac{\a+2}{p})\frac{qs}{q-s}}  |d\eta| \lesssim \epsilon^s
\end{align*}
for all $R> R_0'$.
Therefore, 
\begin{align*}
&\int_{\TT} \sup_{{z}\in\Gamma(\eta)\backslash \ol{D(0,R)}} |g^{(n-k)}(z)|^{\frac{qs}{q-s}}(1-|z|^2)^{(n-k+\frac{\b+2}{t}-\frac{\a+2}{p})\frac{qs}{q-s}}  |d\eta| \\
\lesssim&\int_{\TT} \sup_{{a_j}\in\widetilde{\Gamma}(\eta)\backslash \ol{D(0,R)}} |g^{(n-k)}(a_j)|^{\frac{qs}{q-s}}(1-|a_j|^2)^{(n-k+\frac{\b+2}{t}-\frac{\a+2}{p})\frac{qs}{q-s}}  |d\eta| \lesssim \epsilon^s.
\end{align*}
	
(c) If $0<q=s<\infty$, $0<t<p<\infty$, 
similar to the proof of Theorem \ref{th3}, we have
$$
\sup_{u\in\D}\frac{1}{1-|u|^2}\sum_{a_j\in S(u),|a_j|\ge R}\left(|g^{(n-k)}(a_j)|(1-|a_j|^2)^{n-k+\frac{\b+2}{t}-\frac{\a+2}{p}}\right)^{\frac{pt}{p-t}}(1-|a_j|^2)\lesssim\epsilon^s
$$
for all $R> R_0'$.
Then
$$
\sup_{u\in\D,R> R_0'}\frac{1}{1-|u|^2}\int_{S(u)\backslash\ol{D(0,R)}}|g^{(n-k)}(z)|^{\frac{pt}{p-t}}(1-|z|^2)^{\left( n-k+\frac{\b}{t}-\frac{\a}{p} \right)\frac{pt}{p-t}+1}dA(z)\lesssim\epsilon^s,
$$
which implies that
$$
\lim_{|u|\to 1^{-}}\frac{1}{1-|u|^2}\int_{ S(u)}|g^{(n-k)}(z)|^{\frac{pt}{p-t}}(1-|z|^2)^{\left(n-k+\frac{\b}{t}-\frac{\a}{p}\right){\frac{pt}{p-t}}+1}dA(z)=0.
$$
	
(d) If $0<q<s<\infty$, $0<p,t <\infty$ or $0<q=s<\infty$, $0<p\leq t<\infty$, for $l>0$ and $a\in \D$, define
$$
F_a(z)=\frac{(1-|a|^2)^l}{(1-\ol{a}z)^{l+\frac{\a+2}{p}+\frac{1}{q}}},\quad z\in \D.
$$
Observing that $F_a$ converges to zero  uniformly on any compact subsets of $\D$ as $|a|\to1^-$. The compactness of $T_g^{n,k}$ implies that 
$$
\lim_{|a|\to 1^-}\|T_g^{n,k} F_a\|_{AT_t^s(\b)}=0.
$$
By Lemma \ref{fn}, we obtain
$$
|F_a^{(k)}(z)g^{(n-k)}(z)|\lesssim\frac{\|(T_g^{n,k}F_a)^{(n)}\|_{AT_t^s(\b)}}{(1-|z|^2)^{\frac{\b+2}{t}+\frac{1}{s}+n}}   \quad\text{for all} ~z\in\D.
$$
Replacing $z$ by $a$ in the inequality above, we have
$$
\lim_{|a|\to1^-}|g^{(n-k)}(a)|(1-|a|^2)^{n-k+\frac{\b+2}{t}-\frac{\a+2}{p}+\frac{1}{s}-\frac{1}{q}}=0.
$$
The proof is complete.
\end{proof}

\begin{Theorem}
Let $0<p,q,s<\infty$, $\a>-2$, $g\in H(\D)$, $n\in\NN$  and $k\in\NN\cup\{0\}$ such that $0\le  k< n$ and $(n-k)p-(\a+2)>0$. 
Then $T_g^{n,k}:AT_p^q(\a)\to H^s$ is compact if and only if $T_g^{n,k}:AT_p^q(\a)\to H^s$ is bounded and one of the following conditions holds:

(a) If $0<s<q<\infty$, $2<p<\infty$, 
$g^{(n-k)}\in AT_{\frac{2p}{p-2}}^{\frac{qs}{q-s}}(\delta),$ where $\delta=\left(n-k -1-\frac{\a}{p} \right)\frac{2p}{p-2}$.
	
(b) If $0<s<q<\infty$, $0<p \leq 2$, 
$$
\lim_{R\to 1^{-}}\int_{\TT}\left( \sup_{{z}\in\Gamma(\eta)\backslash \ol{D(0,R)}} |g^{(n-k)}(z)|(1-|z|^2)^{n-k-\frac{\a+2}{p}} \right)^{\frac{qs}{q-s}}|d\eta|=0.
$$
	
(c) If $0<q=s<\infty$, $2<p<\infty$,
$$
\lim_{|u|\to 1^{-}}\frac{1}{1-|u|^2}\int_{ S(u)}|g^{(n-k)}(z)|^{\frac{2p}{p-2}}(1-|z|^2)^{\left(n-k-1-\frac{\a}{p}\right){\frac{2p}{p-2}}+1}dA(z)=0.
$$
	
(d) If $0<q<s<\infty$, $0<p<\infty$ or $0<q=s<\infty$, $0<p\leq 2$,  
$$\lim_{z\to1^-}|g^{(n-k)}(z)|(1-|z|^2)^{n-k-\frac{\a+2}{p}+\frac{1}{s}-\frac{1}{q}}=0.$$
\end{Theorem}

\begin{proof}
The proof of this theorem is similar to the previous theorem and is omitted here.
\end{proof}

	{\bf Data Availability}  No data were used to support this study.\msk
	
	{\bf Conflicts of Interest}  The authors  declare that they have no conflicts of interest.\msk
	
{\bf Authorship contribution statement } Rong Yang and Lian Hu: Writing – review \& editing. Songxiao Li: Writing – original draft, Conceptualization, Supervision. All authors reviewed the manuscript.\msk 

	{\bf Acknowledgements} 
	The corresponding author was supported by  NNSF of China (No. 12371131), STU Scientific Research Initiation Grant(No. NTF23004), 
Li Ka-Shing Foundation (No. 2024 LKSFG06),  Guangdong Basic and Applied Basic Research Foundation (No. 2023A1515010614).\\
	

\end{document}